\documentclass[12pt,reqno]{amsart}

\usepackage{amssymb,amsmath,amsbsy,eucal,amsfonts,mathrsfs,latexsym}
\usepackage{graphicx,verbatim,psfrag,epsfig,enumerate}
\usepackage{stmaryrd}
\usepackage{lscape}
\usepackage{url}
\usepackage{subfig}
\usepackage{multirow}
\usepackage{color}
\usepackage[applemac]{inputenc}

\newtheorem{theorem}{Theorem}[section]
\newtheorem{lemma}[theorem]{Lemma}

\theoremstyle{definition}

\newtheorem{remark}{{\it Remark}}[section]

\numberwithin{equation}{section}

\newcommand{\Vc}{{\mathbf{c}}}

\newcommand{\Vl}{{\mathbf{l}}}

\newcommand{\Vr}{{\mathbf{r}}}

\newcommand{\Vu}{{\mathbf{u}}}

\newcommand{\Ba}{{\boldsymbol{a}}}
\newcommand{\Bb}{{\boldsymbol{b}}}

\newcommand{\Bf}{{\boldsymbol{f}}}

\newcommand{\Bj}{{\boldsymbol{j}}}

\newcommand{\Bn}{{\boldsymbol{n}}}

\newcommand{\Bu}{{\boldsymbol{u}}}
\newcommand{\Bv}{{\boldsymbol{v}}}
\newcommand{\Bw}{{\boldsymbol{w}}}
\newcommand{\Bx}{{\boldsymbol{x}}}

\newcommand{\BA}{{\boldsymbol{A}}}
\newcommand{\BB}{{\boldsymbol{B}}}
\newcommand{\BC}{{\boldsymbol{C}}}
\newcommand{\BD}{{\boldsymbol{D}}}
\newcommand{\BE}{{\boldsymbol{E}}}
\newcommand{\BF}{{\boldsymbol{F}}}

\newcommand{\BH}{{\boldsymbol{H}}}

\newcommand{\BL}{{\boldsymbol{L}}}

\newcommand{\BP}{{\boldsymbol{P}}}

\newcommand{\BV}{{\boldsymbol{V}}}

\newcommand{\BZ}{{\boldsymbol{Z}}}

\newcommand{\Bdelta}{{\boldsymbol{\delta}}}
\newcommand{\BPi}{\boldsymbol{\Pi}}

\newcommand{\curl}{\Vc\Vu\Vr\Vl\,}

\definecolor{red}{rgb}{1,0,0}

\definecolor{blue}{rgb}{0,0,1}

\begin{document}

\title{Analysis of a semi-implicit structure-preserving finite element method for 
the nonstationary incompressible Magnetohydrodynamics equations}

\thanks{The work of Weifeng Qiu is partially 
supported by a grant from the Research Grants Council of the Hong Kong Special Administrative Region, China 
(Project No. CityU 11302718).  The work of Ke Shi is partially supported by Simons Foundation Collaboration Grants for Mathematicians (Award Number: 637267). As a convention the names of the authors are alphabetically ordered. All authors contributed equally in this article. }

\author{Weifeng Qiu}
\address{Department of Mathematics, City University of Hong Kong, 83 Tat Chee Avenue, Hong Kong, China.}
\email{weifeqiu@cityu.edu.hk}
\author{Ke Shi}
\address{Department of Mathematics and Statistics, Old Dominion University, Norfolk, VA 23529, USA.}
\email{kshi@odu.edu}
\subjclass[2000]{Primary: 65N30, 76W05}
\keywords{magnetohydrodynamics, finite element method, structure-preserving}


\begin{abstract}
We revise the structure-preserving finite element method in  [K. Hu, Y. MA and J. Xu. (2017) Stable finite
 element methods preserving $\nabla \cdot \mathbf{B}=0$ exactly for MHD models. Numer. Math., 
 135,  371-396]. The revised method is semi-implicit in time-discretization. We prove the linearized scheme preserves the divergence free property for the magnetic field exactly at each time step. Further, we showed the linearized scheme is unconditionally stable and we obtain optimal convergence in the energy norm of the revised method even for solutions with low regularity.
\end{abstract}
   
\maketitle

\section{Introduction}   
In this paper, we consider the nonstationary incompressible magnetohydrodynamics (MHD) equations over $[0, T] \times \Omega$ where $\Omega\subset \mathbb{R}^{d}$ ($d=2,3$) is a Lipschitz polyhedral domain:
\begin{subequations}\label{MHD_orig}
\begin{align}
\frac{\partial \Bu}{\partial t} + ( \Bu \cdot \nabla) \Bu
- R_{e}^{-1} \Delta \Bu
- S \Bj \times \BB
+ \nabla p
&= \Bf , \\
 \Bj
- R_{m}^{-1} \nabla \times \BB
&= \boldsymbol{0} , \\
 \frac{\partial \BB}{\partial t} + \nabla \times \BE &= \boldsymbol{0} , \\
 \nabla \cdot \BB &= 0, \\
 \nabla \cdot \Bu &= 0, \\
 \Bj &= \BE + \Bu \times \BB, \\
 \intertext{with the boundary and initial conditions as}
 \Bu = 0, \quad \BB \cdot \Bn = 0, \quad \BE \times \Bn = \boldsymbol{0},\quad &\text{on $\partial \Omega$}, \\
 \Bu(\Bx, 0) = \Bu_0(\Bx), \quad \BB(\Bx, 0) = \BB_0(\Bx),
\label{eq:MHD}
\end{align}
\end{subequations}
where $\nabla\cdot \Bu_{0} = \nabla\cdot \BB_{0} = 0$. In (\ref{MHD_orig}), $\Bu$ is the fluid velocity, 
$p$ is the fluid pressure, $\Bj$ is the current density, $\BE$ and $\BB$ are the electric and magnetic fields 
respectively. The system is characterized by three parameters: the hydrodynamic Reynolds number $R_e$, 
the magnetic Reynolds number $R_m$ and the coupling number $S$. $\Bf\in \BL^{2}(\Omega)$ stands for 
the external body force. $\Bn$ denotes the outer unit normal vector on $\partial \Omega$.

The nonstationary incompressible MHD equations have wide applications in fusion reactor blankets 
\cite{Abdou2001}, liquid metals \cite{Gerbeau-LL, Moreau} and plasma physics \cite{Goedbloed-Poedts}. 
The global existence of weak solution is well known. The existence and uniqueness of local strong solutions 
on regular domains is proved in \cite{Sermane-Temam}. There are many research works on numerical methods 
and numerical analysis on the nonstationary incompressible MHD equations. Here we just provide an 
incomplete list \cite{Badia-Codina-Planas, BanasProhl2010, Codina, GaoQiu2019, He2015, Marioni, Prohl}.

Recently, exactly divergence-free discretizations on the magnetic field $\BB$ draws more attentions.  
Though by \cite{GaoQiu2019} it seems that it is tolerable if this property is only satisfied weakly 
in numerical simulations of incompressible MHD equations, we notice that it is desirable to provide 
exactly divergence-free numerical magnetic filed in 
numerical approximations for inductionless MHD model (see \cite{Ni2007a, Ni2007b, Ni2013, Ni2014}).   
Authors of \cite{Hiptmair_LMZ} utilized $H(\text{curl})$-conforming elements to approximate $\BA$ 
which is the potential of $\BB$ ($\BB = \nabla \times \BA$), such that their numerical approximation 
of $\BB$ is exactly divergence-free. It is proved in \cite{Hiptmair_LMZ} that a subsequence of their 
numerical solutions converge to the true solution on any Lipschitz polyhedral domain. 
In \cite{HuMaXu}, a structure-preserving finite element method is developed for the nonstationary 
incompressible MHD equations. Besides $\Bu$ and $\BB$, the electric field $\BE$ is also considered 
as an unknown in the numerical method in \cite{HuMaXu}. By using discretization of the equation 
\[
\dfrac{\partial \BB}{\partial t} + \nabla \times \BE = 0,
\]
the numerical approximation of $\BB$ is exactly divergence-free. Later in \cite{MaXuZhang2016}, 
it is proved that the method in \cite{HuMaXu} achieves optimal convergence in the energy norm under the regularity assumption 
that $\Bj \in L^{\infty}([0, T]; \BL^{\infty}(\Omega))$. 

In this paper, our main contribution is to carefully modify/linearize the structure-preserving finite element method in \cite{HuMaXu} so that it is
semi-implicit with respect to time-discretization and it only need to solve a linear system at each time step. This effort is based on our rigirous analysis of the scheme. In addition, we don't compromise on the accuracy of the method, structure-preserving and/or smoothness of the exact solutions.
We prove optimal convergence for the energy norm even for 
solutions with low regularity. We also show that our numerical approximation of $\BB$ is exactly 
divergence-free and the method is energy conserving. 
 
 The rest of the paper is organized as follows: Section 2 we discribe the linearized scheme together with the main results from our analysis. In Section 3 we present analytic tools needed for the analysis. Details of the proofs for the main result is presented in Section 4.
\section{An implicit linearized mixed FEM}
\subsection{Preliminaries}
In this section, we introduce the notations and spaces that related with the scheme.  We adopt the standard notation for the inner product and the norm of the
{$L^{2}$} space. Namely, for scalar valued functions the inner products are defined as:
$
(u,v):=\int_{\Omega}u\cdot v \mathrm{d}x,\quad
\|u\|:=\left(\int_{\Omega} \lvert u\rvert^2 \mathrm{d}x\right)^{1/2}. 
$
This convention applies to vector and tensor-valued functions as well.
For a function $u \in W^{k,p}(\Omega)$, we use $\|u\|_{k,p}$ for the standard norm in $W^{k,p}(\Omega)$. When $p = 2$ we drop the index $p$, i.e. $\|u\|_k := \|u\|_{k,2}$ and $\|u\|:= \|u\|_{0,2}$. Vector-valued Sobolev spaces, we use the bold version of the corresponding scalar-valued spaces. For instance, $\BH^1(\Omega) := [H^1(\Omega)]^d$. 

In addition to the standard Sobolev spaces over $\Omega$, we define vector function spaces as:
\begin{align*}
\BH(\curl,\Omega)&:=\{\Bv\in L^2(\Omega), \nabla\times \Bv \in [L^2(\Omega)]^3\},  \\
\BH(\text{div},\Omega) &:=\{\Bw \in L^2(\Omega), \nabla\cdot \Bw \in L^2(\Omega)\}, \\
\BH_{0}^1(\Omega) &:=\left\{\Bv\in \BH^{1}(\Omega): \Bv\left |_{\partial \Omega} =0\right .\right \}, \\
\BH_0(\curl,\Omega) &:=\{\Bv\in H(\curl, \Omega), \Bv \times \Bn =0 \mbox{ on } \partial\Omega\}, \\
\BH_0(\text{div},\Omega)&:=\{\Bw\in H(\text{div}, \Omega), \Bw\cdot \Bn=0 \mbox{ on } \partial\Omega\}, \\
\BH(\text{div}0,\Omega)&:=\{\Bw\in H(\text{div}, \Omega), \nabla \cdot \Bw=0 \}, \\
\BH_0(\text{div}0,\Omega)&:=\{\Bw\in H_0(\text{div}, \Omega), \nabla \cdot \Bw=0\}, \\
L^2_0(\Omega) & := \{ q \in L^2(\Omega), \int_{\Omega} q dx  = 0 \}.
\end{align*}



\subsection{The linearized mixed FEM}
Next we introduce some notation and spaces in order to define the linearized mixed FEM for the problem \eqref{MHD_orig}. Let $\mathcal{T}_h$ be a conforming triangulation of the domain $\Omega$ with tetrahedral elements. Here we assume that the triangulation is shape-regular and quasi-uniform. For each element $K \in \mathcal{T}_h$, $h_K$ denotes the diameter of $K$ and the global mesh size is denoted by $h = \max_{K \in \mathcal{T}_h} h_K$. To approximate $(\Bu, p)$, we use the stable pair of Stokes elements $\BV_h \times Q_h \subset \BH^1_0 \times L^2_0(\Omega)$ which satisfies the discerete {\em inf-sup} condition: there exists a constant $\beta > 0$ only depending on $\Omega$ such that
\begin{equation}\label{infsup}
\inf_{q_h \in Q_h \backslash 0} \sup_{\Bv_h \in \BV_h \backslash \boldsymbol{0}} \frac{(q_h, \nabla \cdot \Bv_h)_{\Omega}}{\|\Bv_h\|_{1} \|q_h\|_0} \ge \kappa.
\end{equation}
In this paper, we choose the classical $P^{k+1}$-$P^k$ Taylor-Hood pair:
\begin{align*}
\BV_h &:= \{ v_h \in \BH^1_0(\Omega)| \Bv_h|_{K} \in \BP^{k+1}(K), \forall K \in \mathcal{T}_h\}, \\
Q_h &:= \{q_h \in L^2_0(\Omega) \cap C(\Omega) | q_h|_K \in P^k(K), \forall K \in \mathcal{T}_h\}.
\end{align*}
Here $P^l(K)$ denotes the space of polynomials of degree no more than $l$ over $K$. 

For the other two unknowns $(\BE, \BB)$, we use discrete spaces $\BC_h \times \BD_h \subset \BH_0(\curl, \Omega) \times \BH_0(\text{div}, \Omega)$ which are competible in the sense that they belong to the same finite element de Rham sequence \cite{ArnoldFalkWinther06,ArnoldFalkWinther10}. In this paper, we choose $\BC_h$ to be the $k$-th order second type N\'ed\'elec $\BH(\text{curl})$ element and $\BD_h$ the $k$-th order Brezzi-Douglas-Marini element on simplexes. In this paper we assume $k \geq 1$.

For the time discretization, let $\{t_h\}_{n = 0}^N$ be a uniform partion of time domain $(0, T)$ with the step size $\tau = \frac{T}{N}$, and for generic function $U(\Bx, t)$ we define $U^n = U(\cdot, n \tau)$. Finally, we define
\[
D_{\tau} U^n = \frac{U^{n} - U^{n-1}}{\tau}, \quad \overline{U}^n = \frac{U^n + U^{n-1}}{2}, \quad \text{for} \quad n = 1,2, \cdots, N.
\]

Now we are ready to derive the linearized mixed FEM for the MHD system \eqref{MHD_orig}. For each $n > 0$ we seek approximate solution $(\Bu^n_h, p^n_h, \BE^n_h, \BB^n_h) \in \BV_h \times Q_h \times \BC_h \times \BD_h$ satisfies the following governing equations:
\begin{subequations}\label{MHD_mixed}
\begin{align}
\label{MHD_mixed_a}
(D_{\tau} \Bu^n_h, \Bv) + R_e^{-1} (\nabla \overline{\Bu}^n_h, \nabla \Bv) + \frac12 [(\Bu_h^{n-1} \cdot \nabla \overline{\Bu}_h^n, \Bv) - (\Bu_h^{n-1} \cdot \nabla \Bv, &  \overline{\Bu}_h^n)] \\
\nonumber - S R^{-1}_m((\nabla_h \times \overline{\BB}_h^n) \times \BB_h^{n-1}, \Bv) - (p^n_h, \nabla \cdot \Bv) & = (\Bf^n, \Bv), \\
\label{MHD_mixed_b}
(\Bj^n_h, \BF) - R_m^{-1} (\overline{\BB}^n_h, \nabla \times \BF) &= 0, \\
\label{MHD_mixed_c} (D_{\tau} \BB_h^n, \BZ) + (\nabla \times \BE^n_h, \BZ) & = 0, \\
\label{MHD_mixed_d} (\nabla \cdot \overline{\Bu}^n_h, q) & = 0, \\
\label{MHD_mixed_e} \Bj^n_h = \BE^n_h + \overline{\Bu}^n_h \times &\BB^{n-1}_h, 
\end{align}
\end{subequations}
for all $(\Bv, q, \BF, \BZ) \in \BV_h \times Q_h \times \BC_h \times \BD_h$. At the initial time step, we take $\Bu^0_h = \BPi_V \Bu_0, \BB^0_h = \BPi_D \BB_0$. Here $\BPi_V \Bu_0, \BPi_D \BB_0$ are projections  (defined in the next secion) of the initial data $\Bu_0, \BB_0$ in the spaces $\BV_h, \BD_h$ respectively. 
Here the discrete {\em curl} $(\nabla_h \times \cdot)$ is a linear map $\BL^2(\Omega) \rightarrow \BC_h$ defined as: given $\BB \in \BL^2(\Omega)$, $\nabla_h \times \BB \in \BC_h$ satisfies
\begin{equation}\label{discrete_curl}
(\nabla_h \times \BB, \BF) = (\BB, \nabla \times \BF) \quad \forall \BF \in \BC_h.
\end{equation}

\begin{remark}
Notice that in the above scheme, the convection term and $\Bj^n_h$ are linear with respect to $\Bu^n_h, \BE^n_h$ repectively. Consequently at each time step, the above scheme leads to a linear system for all the unknowns. Here we also want to remark on the fact that in \eqref{MHD_mixed_b} we replaced $\Bj^n_h$ with $R^{-1}_m \nabla_h \times \overline{\BB}^n_h$ comparing with the original scheme defined in \cite{HuMaXu}. This modification requires a global $L^2-$type projection in the assembly process. Nevertheless, from the analysis below we can see that it is crucial to make such modification in order to obtain the desired optimal error estimates. It is not clear if the analysis remains valid if we keep $\Bj^n_h$ in this term.
\end{remark}

\subsection{Main Result}
We first present the stability of the discrete problem \eqref{MHD_mixed} in the following theorem:

\begin{theorem}\label{stability}
The discrete solution $(\Bu^n_h, p^n_h, \BE^n_h, \BB^n_h)$ satisfies
\[
\frac{\|\Bu^n_h\|^2 - \|\Bu^{n-1}_h\|^2}{2\tau} + R^{-1}_e \|\nabla \overline{\Bu}^n_h\|^2 + S R_m^{-2}\|\nabla_h \times \overline{\BB}^n_h\|^2 + S R^{-1}_m \frac{\|\BB^n_h\|^2 - \|\BB^{n-1}_h\|^2}{2\tau} = (\Bf^n, \overline{\Bu}^n_h).
\]
Consequently, we have for $n=1, 2, \dots, N$:
\begin{align*}
\|\Bu^n_h\|^2 + S R^{-1}_m\|\BB^n_h\|^2 + \tau \sum_{i=1}^n & (R_e^{-1} \|\nabla \overline{\Bu}^i_h\|^2 + 2 S R_m^{-2}\|\nabla_h \times \overline{\BB}^i_h\|^2) \\
& \le \|\Bu^0_h\|^2 + S R^{-1}_m\|\BB^0_h\|^2 + C \tau \sum_{i=1}^n R_e \|\Bf_i\|^2_{-1}. 
\end{align*}

In addition, the magmetic field is exactly divergence free:
\[
\nabla \cdot \BB^n_h = 0, \quad \for n=1,2,\dots, N,
\] 
provided $\nabla \cdot \BB^0_h = 0$.
\end{theorem}

\begin{proof}
Taking $(\Bv, \BF, \BZ, q) = (\overline{\Bu}^n_h, -S R^{-1}_m \nabla_h \times \overline{\BB}^n_h, S R^{-1}_m \overline{\BB}^n_h, p^n_h)$ in \eqref{MHD_mixed_a} - \eqref{MHD_mixed_d} and adding together, after some algebraic simplification we have:
\begin{align*}
\frac{\|\Bu^n_h\|^2 - \|\Bu^{n-1}_h\|^2}{2\tau} &+ R^{-1}_e \|\nabla \overline{\Bu}^n_h\|^2 + S R^{-2}_m \|\nabla_h \times \overline{\BB}^n_h\|^2 + S R^{-1}_m \frac{\|\BB^n_h\|^2 - \|\BB^{n-1}_h\|^2}{2\tau} = (\Bf^n, \overline{\Bu}^n_h) \\
& \le C \|\Bf^n\|_{-1} \|\nabla \overline{\Bu}^n_h\| \le C R_e \|\Bf^n\|^2_{-1} + \frac12 R^{-1}_e \|\nabla \overline{\Bu}^n_h\|^2. 
\end{align*}
In the above estimate we used the Cauchy-Schwartz inequality, Poincar\'e inequality and Young's inequality. Hence for any $n = 1, 2, \dots, N$ if we sum over the above estimate from $1$ to $k$ we have
\begin{align*}
\|\Bu^n_h\|^2 + S R^{-1}_m\|\BB^n_h\|^2 + \tau \sum_{i=1}^n & (R_e^{-1} \|\nabla \overline{\Bu}^i_h\|^2 + 2 S R_m^{-2}\|\nabla_h \times \overline{\BB}^i_h\|^2) \\
& \le \|\Bu^0_h\|^2 + S R^{-1}_m\|\BB^0_h\|^2 + C \tau \sum_{i=1}^n R_e \|\Bf_i\|^2_{-1}. 
\end{align*}
This completes the proof for the first assersion. For the second part, notice that $\nabla \times \BC_h \subset \BD_h \cap \BH_0(\text{div}0,\Omega)$. Hence \eqref{MHD_mixed_c} is equivalent as
\[
D_{\tau} \BB^n_h + \nabla \times \BE^n_h = 0.
\]
Or 
\[
\frac{\BB^n_h - \BB^{n-1}_h}{\tau} + \nabla \times \BE^n_h = 0.
\]
Taking the divergence of the above equation we have:
\[
\nabla \cdot (\BB^n_h - \BB^{n-1}_h) = 0.
\]
This completes the proof.
\end{proof}

For the error estimates, we assume that the exact solution of MHD system \eqref{MHD_orig} uniquely exists and the unknowns have following regularity property:
\begin{equation}\label{regularity}
\begin{aligned}
&\Bu \in L^{\infty}(0,T; \BH^{1+s}(\Omega)), \Bu_t \in L^2(0, T; \BH^{1 + s}), \Bu_{tt} \in L^{2}(0,T; \BL^2(\Omega)); \\
&p \in L^{\infty}(0,T; H^s(\Omega)), p_t \in L^2(0,T; H^s(\Omega)); \\
&\BB, \nabla \times \BB \in L^{\infty}(0, T; \BH^s(\Omega)), \BB_t, \nabla \times \BB_t, \BB_{tt} \in L^2(0,T; L^2(\Omega)) \\
&\BE, \nabla \times \BE \in L^{\infty}(0,T; H^s(\Omega)),
\end{aligned}
\end{equation}
where $s > \frac12$. Under this assumption, our main error estimate result can be summarized as follows:

\begin{theorem}\label{main_error}
Let $(\Bu, p, \BB, \BE)$ be the exact solution of \eqref{MHD_orig} with the above regularity \eqref{regularity} holds. Let $(\Bu_h, p_h, \BB_h, \BE_h)$ be the numerical solution of the discrete system \eqref{MHD_mixed}. Then we have for all $n = 1, 2, \cdots, N$
\begin{align}\label{energy_all}
\|\Bu^n -\Bu^n_h\|^2 + \|\BB^n - \BB^n_h\|^2 &+ C \tau \sum_{j = 1}^n ( \|\nabla \overline{\Bu}^n - \nabla \overline{\Bu}^n_h\|^2 + \|\nabla_h \times \overline{\BB}^n - \nabla_h \times \overline{\BB}^n_h\|^2) \\
\nonumber
& \le e^{2CT} (h^{2\beta} + \tau^2),
\end{align}
at each time step, we also have
\begin{equation}\label{est_Bu_tn}
\|\nabla {\Bu}^n - \nabla {\Bu}^n_h\|^2 + \|\nabla_h \times {\BB}^n - \nabla_h \times {\BB}^n_h\|^2 \le C (h^{2\beta} + \tau^2).
\end{equation}
with $\beta = \min \{s, k+1\}$ and $C$ depends on the physical parameters but is independent of the discrete paramters $\tau$ and $h$. Further, at each time step, we have
\begin{equation}\label{estimate_e1}
\|\BE^n - \BE^n_h\|^2 \le C (\tau + h^{2\beta}).
\end{equation}
\begin{equation}\label{estimate_p}
\|p^n - p^n_h\|^2 \le C (\tau^{-1} h^{2\beta} + \tau).
\end{equation}
If we further assume that $\Bu_t \in L^{\infty}(0, T; \BH^1(\Omega)); \BB_t, \nabla \times \BB_t \in L^{\infty}(0, T; \BL^2(\Omega))$, we have that:
\begin{equation}\label{estimate_e2}
\|\BE^n - \BE^n_h\| \le C (\tau^2 + h^{2\beta}).
\end{equation}
\end{theorem}

\section{Auxiliary estimates}
In this section, we gather the necessary tools for the final error estimates in the next section. First we present an approximation property for the discrete {\em curl} operator:

\begin{lemma}\label{curl_h}
For any vector field $\BC \in \BH(\text{curl}, \Omega)$, we have
\[
\|\nabla_h \times \BC\|_{L^p(\Omega)} \le \|\nabla \times \BC\|_{L^p(\Omega)}, 
\]
with any $p \in (1, \infty)$.
\end{lemma}

\begin{proof}
Define $\BPi_h: \BL^2(\Omega) \rightarrow \BC_h$ be the standard $L^2$-projection. By the definition of 
the discrete {\em curl} operator \eqref{discrete_curl} we have for any $\BC \in \BH(\text{curl}, \Omega)$
\[
(\nabla_h \times \BC, \BF) = (\BC, \nabla \times \BF) = (\nabla \times \BC, \BF) \quad \; \forall \BF \in \BC_h.
\]
Therefore, this implies that $\nabla_h \times \BC = \BPi_h (\nabla \times \BC)$. 
Similar to the proof of \cite[Theorem~$3$]{CT1987}, we have:
\[
\|\nabla_h \times \BC\|_{L^p(\Omega)} = \|\BPi_h (\nabla \times \BC)\|_{L^p(\Omega)} 
\le C_{p} \|\nabla \times \BC\|_{L^p(\Omega)}, 
\]
with any $p \in [1, +\infty]$.
\end{proof}

The next result gathers classical and discrete Sobolev inequalities needed for the error estimates in the next section \cite{Nirenberg1966,Hiptmair2002}.

\begin{lemma}\label{sobolev_embed}
For $\Bu \in \BH^{1 + s}(\Omega)$ with $s > \frac12$ we have
\begin{align*}
\|\Bu\|_{0,p} &\le C \|\Bu\|_1, \quad \text{for} \quad 1 \le p \le 6, \\
\|\Bu\|_{0, \infty} &\le C \|\Bu\|_{1+s}.
\end{align*}
For $\BB \in \BH^{s}(\Omega)$ with $s > \frac12$, we have 
\[
\|\BB\|_{0,3} \le C \|\BB\|_s.
\]
Further, for $\BB \in \BH^{s}(\Omega) \cap \BH(\text{div}0,\Omega)$, we have
\[
\|\BB\|_{0,3} \le C \|\BB\|_s \le C \|\nabla \times \BB\|.
\]
\end{lemma} 

Next we define the projections of the unknowns $(\BPi_V \Bu, \Pi_Q p, \BPi_C \BE, \BPi_D \BB )$ and gather their approximation properties. 
For the fluid pair $\Bu, p$, we follow the idea used in \cite{GaoQiu2019}. Namely, for a fixed $t \in (0, T]$, for the exact solution $(\Bu, p) \in \BH^1_0(\Omega) \times L^2_0(\Omega)$ we define the Stokes projection $(\BPi_V \Bu, \Pi_Q p) \in \BV_h \times Q_h$ satisfies
\begin{subequations}\label{Stokes_proj}
\begin{align}
R^{-1}_e (\nabla \BPi_V \Bu, \nabla \Bv) - (\Pi_Q p, \nabla \cdot \Bv) &= R^{-1}_e (\nabla \Bu, \nabla \Bv) - (p, \nabla \cdot \Bv), \\
(\nabla \cdot \BPi_V \Bu, q) & = (\nabla \cdot \Bu, q), 
\end{align}
\end{subequations}
for all $(\Bv, q) \in \BV_h \times Q_h$. We can see that the above projection is defined globally over $\Omega$ through the variational form of Stokes equations. 
For the electric field $\BE$ we simply use the Ned\'el\'ec $\BH$-curl projection  \cite{Monk2003}, denoted by $\BPi_C E$.

Finally, for the magnetic field $\BB$, notice that $\BB_h \in \BD^0_h := \BD_h \cap \BH_0(\text{div}0, \Omega)$ and $\BB \in \BH_0(\text{div}0, \Omega)$. We define the $L^2$-projection $\BPi_D: \BL^2(\Omega) \rightarrow \BD^0_h$ such that $\BPi_D \BB \in \BD^0_h$ satisfies:
\begin{equation}\label{div_proj}
(\BPi_D \BB, \BZ) = (\BB, \BZ) \quad \forall \BZ \in \BD^0_h.
\end{equation}

We have the following approximation property result for the projections \cite{Girault.V;Raviart.P.1986a, GaoQiu2019}:
\begin{lemma}\label{proj_prop}
Under the regularity assumption \eqref{regularity}, the above projection satisfies
\begin{align*}
\|\Bdelta_u\|_1 + \|\Bdelta_p\| &\le C h^\beta (\|\Bu\|_{1+\beta} + \|p\|_\beta), \\
\|\frac{\partial \Bdelta_u}{\partial t}\|_1 &\le C h^\beta (\|\Bu_t\|_{1+\beta} + \|p_t\|_\beta),\\
\|\BPi_V \Bu\|_{\infty} + \|\BPi_V \Bu\|_{1,3} &\le C(\|\Bu\|_{1+\beta} + \|p\|_\beta) < \infty, \\
\|\Bdelta_E\| + \|\nabla \times \Bdelta_E\| &\le C h^s (\|\BE\|_\beta + \|\nabla \times \BE\|_\beta), \\
\|\Bdelta_B\| &\le C h^\beta \|\BB\|_\beta, \\
\nabla_h \times \Bdelta_B &= \boldsymbol{0}.
\end{align*}
with $\beta = \min \{s, k+1\}$.
\end{lemma}

\begin{proof}
It suffice to establish the last two inequality and identity since others are well-known results \cite{Girault.V;Raviart.P.1986a}. Notice that since $\BB \in \BH_0(\text{div}0, \Omega)$ we have that its BDM projection $\BPi_{\text{BDM}}\BB \in \BD^0_h$. This implies that 
\begin{equation}\label{approx_proj_B}
\|\Bdelta_B\| \le \|\BB - \BPi_{\text{BDM}} \BB\| \le C h^{\beta} \|\BB\|_\beta.
\end{equation}
For the last identity, we can derive this identity by the definition of ``$\nabla_h \times$'' \eqref{discrete_curl} and the projection $\BPi_D$ is $L^2$-projection onto $\BD^0_h$: for any $\BF \in \BC_h$
\[
(\nabla_h \times \BPi_D \BB, \BF)  = (\BPi_D \BB, \nabla \times \BF) = (\BB, \nabla \times \BF) = (\nabla_h \times \BB, \BF).
\]
This completes the proof since $\nabla_h \times \BB, \nabla_h \times \BPi_D \BB \in \BC_h$.
\end{proof}

As a consequence of the above result, we have that the intial errors satisfy:
\begin{equation}\label{initial_eror}
\|\Bu^0 - \Bu^0_h\|_1 + \|\BB^0 - \BB^0_h\| \le C h^{\beta}, \quad \nabla_h \times \BB^0 - \nabla_h \times \BB^0_h = \boldsymbol{0}.
\end{equation}  
Finally, we need the well-known discrete Gronwall's inequality \cite{HeywoodRannacher1990}:

\begin{lemma}\label{discrete_gronwall}
Let $\tau, B$ and $a_k, b_k, c_k, \gamma_k$ be non-negative numbers for all integers $k \ge 0$, 
\[
a_J + \tau \sum^J_{k = 0} b_k \le \tau \sum^J_{k =0} \gamma_k a_k + \tau \sum^J_{k= 0} c_k + B, \quad \text{for} \quad J \ge 0,
\]
suppose that $\tau \gamma_k < 1$ for all $k$ and set $\sigma_k = (1 - \tau \gamma_k)^{-1}$, Then it holds:
\[
a_J + \tau \sum^J_{k = 0} b_k \le e^{\tau \sum^J_{k=0} \gamma_k \sigma_k} (\tau \sum^J_{k=0} c_k + B).
\]
\end{lemma}

\section{Error Estimates}
In this section we present the main error estimates of the method. We first carry out the error equations for the error estimates. By convention, for a generic unknown $\mathcal{U}$, its numerical approximation $\mathcal{U}_h$ and its projection $\Pi \mathcal{U}$, we split the errors as:
\begin{equation}\label{err_split}
\mathcal{U} -\mathcal{U}_h = (\mathcal{U} - \Pi \mathcal{U}) + (\Pi \mathcal{U} - \mathcal{U}_h) := e_{\mathcal{U}} + \delta_{\mathcal{U}}.
\end{equation}

First we notice that the exact solution of the system \eqref{MHD_orig} satisfies the following variational equations at time $t_n$: 
\begin{subequations}\label{trunc_eq}
\begin{align}
(D_{\tau} \Bu^n, \Bv) + R_e^{-1} (\nabla \overline{\Bu}^n, \nabla \Bv) + \frac12 [(\Bu^{n-1} \cdot \nabla \overline{\Bu}^n, \Bv) - (\Bu^{n-1} \cdot \nabla \Bv, &  \overline{\Bu}^n)] \\
\nonumber - S R^{-1}_m((\nabla_h \times \overline{\BB}^n) \times \BB^{n-1}, \Bv) - (p^n, \nabla \cdot \Bv) & = (\Bf^n, \Bv) + \mathcal{R}_1(\Bv), \\
(\tilde{\Bj}^n, \BF) - R_m^{-1} (\overline{\BB}^n, \nabla \times \BF) &= \mathcal{R}_2(\BF), \\
 (D_{\tau} \BB^n, \BZ) + (\nabla \times \BE^n, \BZ) & = \mathcal{R}_3(\BZ), \\
(\nabla \cdot \overline{\Bu}^n, q) & = 0, \\
\tilde{\Bj}^n = \BE^n + \overline{\Bu}^n \times &\BB^{n-1}, 
\end{align}
\end{subequations}
for all $(\Bv, q, \BF, \BZ) \in \BV_h \times Q_h \times \BC_h \times \BD_h$. Here $\mathcal{R}_1, \mathcal{R}_2, \mathcal{R}_3$ are the truncation error terms as follows:
\begin{align*}
\mathcal{R}_1(\Bv) = & (D_{\tau} \Bu^n - \Bu^n_t, \Bv) + Re^{-1} (\nabla \overline{\Bu}^n - \nabla \Bu^n, \nabla \Bv) \\
&+ \frac12 [(\Bu^{n-1} \cdot \nabla \overline{\Bu}^n, \Bv) - (\Bu^{n-1} \cdot \nabla \Bv, \overline{\Bu}^n)] - (\Bu^n \cdot \nabla \Bu^n, \Bv) \\
& - S R^{-1}_m[((\nabla_h \times \overline{\BB}^n) \times \BB^{n-1}, \Bv) - ((\nabla \times {\BB}^n) \times \BB^{n}, \Bv)], \\
\mathcal{R}_2(\BF) = & (\overline{\Bu}^n \times \BB^{n-1} - \Bu^n \times \BB^n, \BF) - R_m^{-1} (\overline{\BB}^n - \BB^n, \nabla \times \BF), \\
\mathcal{R}_3(\BZ) = &  (D_{\tau} \BB^n - \BB^n_t, \BZ).
\end{align*}

If we subtract the numerical system \eqref{MHD_mixed} from the above system \eqref{trunc_eq}, with some algebraic simplification and the projection properties \eqref{Stokes_proj}, \eqref{div_proj} we can obtain the error equations as follows:
\begin{lemma}
The projection errors $(e_{\Bu}, e_p, e_{\BE}, e_{\BB})$ satisfies the system: 
\begin{subequations}\label{error_eq}
\begin{align}\label{error_eq1}
(D_{\tau} e^n_{\Bu}, \Bv) + R_e^{-1} (\nabla \overline{e}_\Bu^n, \nabla \Bv) - (e_p^n, \nabla \cdot \Bv)   & = - (D_{\tau} \Bdelta^n_{\Bu}, \Bv) + \mathcal{R}_1(\Bv)+ \mathcal{O}(\Bv) + \mathcal{M}_1(\Bv)\\
\label{error_eq2}
(e_{\BE}^n, \BF) - R_m^{-1} (\overline{e}_\BB^n, \nabla \times \BF) &= - (\Bdelta_{\BE}^n, \BF) + \mathcal{R}_2(\BF) - \mathcal{M}_2(\BF), \\
\label{error_eq3}
 (D_{\tau} e_\BB^n, \BZ) + (\nabla \times e_\BE^n, \BZ) & = -  (D_{\tau} \Bdelta_\BB^n, \BZ) - (\nabla \times \Bdelta^n_{\BE}, \BZ) + \mathcal{R}_3(\BZ), \\
\label{error_eq4}
(\nabla \cdot \overline{e}_\Bu^n, q) & = 0, 
\end{align}
\end{subequations}
for all $(\Bv, q, \BF, \BZ) \in \BV_h \times Q_h \times \BC_h \times \BD_h$. Here the nonlinear terms are gathered as:
\begin{align*}
\mathcal{O}(\Bv) &= - \frac12 [(\Bu^{n-1} \cdot \nabla \overline{\Bu}^n, \Bv) - (\Bu^{n-1} \cdot \nabla \Bv,  \overline{\Bu}^n)] + \frac12 [(\Bu_h^{n-1} \cdot \nabla \overline{\Bu}_h^n, \Bv) - (\Bu_h^{n-1} \cdot \nabla \Bv, &  \overline{\Bu}_h^n)]\\
\mathcal{M}_1(\Bv) & = S R^{-1}_m((\nabla_h \times \overline{\BB}^n) \times \BB^{n-1}, \Bv) - S R^{-1}_m((\nabla_h \times \overline{\BB}_h^n) \times \BB_h^{n-1}, \Bv), \\
\mathcal{M}_2(\BF) &= (\overline{\Bu}^n \times \BB^{n-1} - \overline{\Bu}^n_h \times \BB^{n-1}_h, \BF).
\end{align*}
\end{lemma}

We are ready to prove our main result Theorem \ref{main_error} with the above error equations.

\begin{proof} of Theorem \ref{main_error}
We start by taking $(\Bv, \BF, \BZ, q) = (\overline{e}^n_{\Bu}, -S R^{-1}_m \nabla_h \times \overline{e}^n_{\BB}, S R^{-1}_m \overline{e}^n_{\BB}, e^n_{p})$ in the error equations \eqref{error_eq1} - \eqref{error_eq4} and adding togather, with some algebraic simplification we have:
\begin{align*}
\frac{\|e^n_{\Bu}\|^2 - \|e^{n-1}_{\Bu}\|^2}{2 \tau} &+ R^{-1}_e \|\nabla \overline{e}^n_{\Bu}\|^2 + S R^{-1}_m \frac{\|e^n_{\BB}\|^2 - \|e^{n-1}_{\BB}\|^2}{2 \tau} + S R^{-2}_m \|\nabla_h \times \overline{e}^n_{\BB}\|^2 \\
&\hspace{-1cm} =  - (D_{\tau} \Bdelta^n_{\Bu}, \overline{e}^n_{\Bu}) + S R^{-1}_m  (\Bdelta_{\BE}^n, \nabla_h \times \overline{e}^n_{\BB}) -  S R^{-1}_m (D_{\tau} \Bdelta_\BB^n,  \overline{e}^n_{\BB}) -  S R^{-1}_m (\nabla \times \Bdelta^n_{\BE}, \overline{e}^n_{\BB}) \\
& \hspace{-1cm}+ \mathcal{R}_1(\overline{e}^n_{\Bu}) + \mathcal{R}_2(- S R^{-1}_m \nabla_h \times \overline{e}^n_{\BB}) + \mathcal{R}_3(\overline{e}^n_{\BB}) \\
& \hspace{-1cm}+ \mathcal{O}(\overline{e}^n_{\Bu}) + \mathcal{M}_1(\overline{e}^n_{\Bu}) + \mathcal{M}_2 (S R^{-1}_m \nabla_h \times \overline{e}^n_{\BB}). 
\end{align*}
Next we will estimate each term on the right hand side of the above identity. For the first four linear terms, we simply use the Cauchy-Schwarz inequality and the approximation property of the projections Lemma \ref{proj_prop} as follows:
\begin{align*}
 (D_{\tau} \Bdelta^n_{\Bu}, \overline{e}^n_{\Bu}) &= \frac{1}{\tau} \int_{t_{n-1}}^{t_n} \int_{\Omega}  \frac{\partial \Bdelta_u}{\partial t}(\rho, \Bx) \cdot \overline{e}^n_{\Bu} d\Bx d\rho \le \frac{1}{\tau} \int_{t_{n-1}}^{t_n} \|\frac{\partial \Bdelta_{\Bu}}{\partial t}(\rho, \cdot)\| \|\overline{e}^n_{\Bu}\| d\rho \\
 & \le \frac{h^\beta\|\overline{e}^n_{\Bu}\|}{\tau} \int_{t_{n-1}}^{t_n} \|\Bu_t(\rho, \cdot)\|_\beta d\rho \\
 & \le \|\overline{e}^n_{\Bu}\|^2 +  \frac{h^{2\beta}}{\tau^2} \int_{t_{n-1}}^{t_n} \|\Bu_t(\rho, \cdot)\|^2_\beta d\rho \int_{t_{n-1}}^{t_n} 1 d\rho \\
& \le \|\overline{e}^n_{\Bu}\|^2 +  \frac{h^{2\beta}}{\tau} \|\Bu_t(\rho, \cdot)\|^2_{L^2(t_{n-1}, t_n; \BH^\beta(\Omega))}.
\end{align*}

For the second linear term we simply apply Cauchy-Schwarz inequality to have
\[
 S R^{-1}_m  (\Bdelta_{\BE}^n, \nabla_h \times \overline{e}^n_{\BB}) \le C h^\beta \|\BE\|_\beta \|\nabla_h \times \overline{e}^n_{\BB}\| \le C \epsilon \|\nabla_h \times \overline{e}^n_{\BB}\|^2 + C \epsilon^{-1} h^{2\beta}.
\]

For the third linear term, since $\Pi_D \BB, \BB_h \in D_h^0$, with the orthogonal property of $\Pi_D$ we have
\[
S R^{-1}_m (D_{\tau} \Bdelta_\BB^n,  \overline{e}^n_{\BB}) = 0.
\]
For the last linear term, we have
\[
 S R^{-1}_m (\nabla \times \Bdelta^n_{\BE}, \overline{e}^n_{\BB}) \le C h^\beta\|\nabla \times \BE\|_\beta \|\overline{e}^n_{\BB}\| \le C \|\overline{e}^n_{\BB}\|^2 + C h^{2\beta}.
\]

Truncation error estimates:
In $R_1(\Bv)$, there is a term as:
\[
(D_{\tau}\Bu^{n} - \Bu^n_t , \overline{e}^n_{\Bu}),
\]
here is how we estimate this term by $\|\Bu_{tt}\|_{L^2(t_{n-1}, t_n;\BL^2(\Omega))}$:
\begin{align*}
(D_{\tau}\Bu^{n} - \Bu^n_t , \overline{e}^n_{\Bu}) & = \frac{1}{\tau}\int_{\Omega} \int_{t_{n-1}}^{t_n} \Bu_t(\rho, \cdot) - \Bu_t(t_n, \cdot) d\rho \cdot \overline{e}^n_{\Bu} d \Bx \\
&= \frac{1}{\tau} \int_{\Omega} \int_{t_{n-1}}^{t_n} \int_{t_{n-1}}^\rho \Bu_{tt}(\sigma, \cdot) \overline{e}^n_{\Bu} d \sigma d\rho d \Bx \\
&\le \frac{1}{\tau} \int_{t_{n-1}}^{t_n} \int_{t_{n-1}}^\rho \|\Bu_{tt}(\sigma, \cdot)\| \|\overline{e}^n_{\Bu}\|d\sigma d\rho = \frac{1}{\tau}  \|\overline{e}^n_{\Bu}\| \int_{t_{n-1}}^{t_n} \int_{t_{n-1}}^\rho \|\Bu_{tt}(\sigma, \cdot)\| d\sigma d\rho \\
&\le  \frac{1}{\tau}  \|\overline{e}^n_{\Bu}\| \int_{t_{n-1}}^{t_n} \|\Bu_{tt}(\sigma, \cdot)\|_{L^2(t_{n-1}, \rho;\BL^2(\Omega))} \|1\|_{L^2(t_{n-1},\rho)} d\rho\\
&\le  \frac{1}{\tau}  \|\overline{e}^n_{\Bu}\| \int_{t_{n-1}}^{t_n} \|\Bu_{tt}(\sigma, \cdot)\|_{L^2(t_{n-1}, t_n;\BL^2(\Omega))} (\rho - t_{n-1})^{\frac12} d\rho \\
& \le C \tau \|\Bu_{tt}\|^2_{L^2((t_{n-1}, t_n), \BL^2(\Omega))} + C \|\overline{e}^n_{\Bu}\|^2.
\end{align*}

Similarly, for $\mathcal{R}_3(S R^{-1}_m \overline{e}^{n}_{\BB})$ we have the following estimates:
\begin{equation}\label{R_3_estimates}
\mathcal{R}_3(S R^{-1}_m \overline{e}^{n}_{\BB}) \le C \tau \|\BB_{tt}\|_{L^2(t_{n-1}, t_n; \BL^2(\Omega))} + C \|\overline{e}^n_{\BB}\|^2.
\end{equation}

For other terms in $\mathcal{R}_1(\overline{e}^n_{\Bu})$, the estimates are similar as the one shown below in $\mathcal{R}_2(- S R^{-1}_m \nabla_h \times \overline{e}^n_{\BB})$ and we gather the result as follows:
\begin{align*}
R^{-1}_e (\nabla \overline{\Bu}^n &- \nabla \Bu^n, \nabla\overline{e}^n_{\Bu}) \le C \epsilon \|\nabla \overline{e}^n_{\Bu}\|^2 + C \epsilon^{-1} \tau \|\nabla \Bu_t\|^2_{L^2(t_{n-1}, t_n; \BL^2(\Omega))}, \\
\frac12 [(\Bu^{n-1} \cdot \nabla \overline{\Bu}^n, \overline{e}^n_{\Bu}) &- (\Bu^{n-1} \cdot \nabla \overline{e}^n_{\Bu}, \overline{\Bu}^n)] - (\Bu^n \cdot \nabla \Bu^n, \overline{e}^n_{\Bu}) \\
&\hspace{-1cm} \le C \epsilon \|\nabla \overline{e}^n_{\Bu}\|^2 + C \|\overline{e}^n_{\Bu}\|^2 + C \epsilon^{-1}\tau (\|\Bu_t\|^2_{L^2(t_{n-1}, t_n; \BL^2(\Omega))} + \|\nabla \Bu_t\|^2_{L^2(t_{n-1}, t_n; \BL^2(\Omega))}) \\
- S R^{-1}_m[((\nabla_h \times \overline{\BB}^n) &\times \BB^{n-1}, \overline{e}^n_{\Bu}) - ((\nabla \times {\BB}^n) \times \BB^{n}, \overline{e}^n_{\Bu})] \\
& \le C \epsilon \|\nabla \overline{e}^n_{\Bu}\|^2 + C \epsilon^{-1}\tau (\|\BB_t\|^2_{L^2(t_{n-1}, t_n; \BL^2(\Omega))} + \|\nabla \times \BB_t\|^2_{L^2(t_{n-1}, t_n; \BL^2(\Omega))}).
\end{align*}
Combining the estimates for all the terms in $\mathcal{R}_1(\overline{e}^n_{\Bu})$ we have
\begin{equation}\label{R_3}
\begin{aligned}
\mathcal{R}_1(\overline{e}^n_{\Bu}) &\le C \epsilon \|\nabla \overline{e}^n_{\Bu}\|^2 + C \| \overline{e}^n_{\Bu}\|^2 + C h^{2s}\|\nabla \times \overline{\BB}^n\|_s + C \tau \|\Bu_{tt}\|^2_{L^2((t_{n-1}, t_n), \BL^2(\Omega))}\\
& + C \epsilon^{-1}\tau (\|\Bu_t\|^2_{L^2(t_{n-1}, t_n; \BH^1(\Omega))}+ \|\BB_t\|^2_{L^2(t_{n-1}, t_n; \BL^2(\Omega))} + \|\nabla \times \BB_t\|^2_{L^2(t_{n-1}, t_n; \BL^2(\Omega))}).
\end{aligned}
\end{equation}

In $\mathcal{R}_2(- S R^{-1}_m \nabla_h \times \overline{e}^n_{\BB})$, there is a term like this: (omit the coefficient for simplicity)
\begin{align*}
(\overline{\Bu}^n \times \BB^{n-1} - \Bu^n \times \BB^n, \nabla_h \times \overline{e}^n_{\BB}) & = (\overline{\Bu}^n \times (\BB^{n-1} - \BB^n) + (\overline{\Bu}^n - \Bu^n) \times \BB^n, \nabla_h \times \overline{e}^n_{\BB}) \\
&= T_1 + T_2.
\end{align*}
For $T_1$, 
\begin{align*}
T_1 &=  \int_{t_{n-1}}^{t_n} \int_{\Omega} \BB_t(\rho, \cdot) \times \overline{\Bu}^n \cdot \nabla_h \times \overline{e}^n_{\BB} d \Bx d\rho \le \int_{t_{n-1}}^{t_n}\|\BB_t(\rho, \cdot)\| \|\overline{\Bu}^n\|_{L^{\infty}(\Omega)} \| \nabla_h \times \overline{e}^n_{\BB} \| d \rho \\
& \le \|\Bu\|_{L^{\infty}(t_{n-1}, t_n; \BL^{\infty}(\Omega))} \| \nabla_h \times \overline{e}^n_{\BB} \| \int_{t_{n-1}}^{t_n}\|\BB_t(\rho, \cdot)\| d \rho \\
& \le \epsilon \| \nabla_h \times \overline{e}^n_{\BB} \|^2 + \epsilon^{-1} \|\Bu\|^2_{L^{\infty}(t_{n-1}, t_n; L^{\infty}(\Omega))} (\int_{t_{n-1}}^{t_n}\|\BB_t(\rho, \cdot)\| d \rho)^2\\
& \le \epsilon \| \nabla_h \times \overline{e}^n_{\BB} \|^2 + C \epsilon^{-1} \tau \|\BB_t\|^2_{L^2(t_{n-1}, t_n; \BL^2(\Omega))}.
\end{align*}
The last step we used the Cauchy-Schwarz inequality and the fact that $\overline{\Bu}^n \in \BH^{1 + s}(\Omega) \hookrightarrow \BL^{\infty}(\Omega)$.
For $T_2$, with a similar technique as above, we have 
\begin{align*}
T_2 &\le \epsilon \| \nabla_h \times \overline{e}^n_{\BB} \|^2 + \epsilon^{-1} \tau \|\BB^n\|^2_{\BL^{3}(\Omega)} \|\Bu_t\|^2_{L^2(t_{n-1}, t_n; \BH^1(\Omega))} \\
& \le \epsilon \| \nabla_h \times \overline{e}^n_{\BB} \|^2 + C \epsilon^{-1} \tau \|\Bu_t\|^2_{L^2(t_{n-1}, t_n; \BH^1(\Omega))},
\end{align*}
the last step we used the regularity assumption \ref{regularity} and that $\BH^s(\Omega) \hookrightarrow \BL^3(\Omega)$.

For the second term in $\mathcal{R}_2(-SR^{-1}_m \nabla_h \times \overline{e}^n_{\BB})$, we have
\begin{align*}
S R^{-2}_m (\overline{\BB}^n - \BB^n, \nabla \times  (\nabla_h \times \overline{e}^n_{\BB})) & = -\frac12 S R^{-2}_m (\nabla \times ({\BB}^n - \BB^{n-1}), \nabla_h \times \overline{e}^n_{\BB}) \\
& =  -\frac12 S R^{-2}_m \int_{t_{n-1}}^{t_n} \int_{\Omega} \nabla \times \BB_t(\rho, \Bx) \cdot \nabla_h \times \overline{e}^n_{\BB} d\Bx d\rho \\
&\le C   \| \nabla_h \times \overline{e}^n_{\BB}\| \int_{t_{n-1}}^{t_n}  \|\nabla \times \BB_t(\rho, \Bx)\| d\rho \\
& \le C (\epsilon  \| \nabla_h \times \overline{e}^n_{\BB}\|^2 + \epsilon^{-1} (\int_{t_{n-1}}^{t_n}  \|\nabla \times \BB_t(\rho, \Bx)\| d\rho)^2) \\
& \le C \epsilon  \| \nabla_h \times \overline{e}^n_{\BB}\|^2 + C \epsilon^{-1} \int_{t_{n-1}}^{t_n}  \|\nabla \times \BB_t(\rho, \Bx)\|^2 d\rho \int_{t_{n-1}}^{t_n}  1 d\rho \\
& = C \epsilon  \| \nabla_h \times \overline{e}^n_{\BB}\|^2 + C \epsilon^{-1} \tau  \|\nabla \times \BB_t\|^2_{L^2(t_{n-1}, t_n; \BL^2(\Omega))}.
\end{align*}
Combine above estimates, we have
\begin{align}\label{R_2}
\mathcal{R}_2 (-S R^{-2}_m &\nabla_h \times \overline{e}^n_{\BB}) \le C \epsilon  \| \nabla_h \times \overline{e}^n_{\BB}\|^2 \\
\nonumber
&+ {C} \epsilon^{-1} \tau (\|\BB_t\|^2_{L^2(t_{n-1}, t_n; \BL^2)(\Omega)} + \|\nabla \times \BB_t\|^2_{L^2(t_{n-1}, t_n; \BL^2(\Omega))} + \|\Bu_t\|^2_{L^2(t_{n-1}, t_n; \BL^2(\Omega))}).
\end{align}

Finally, we bound the nonlinear terms as follows:
\begin{align*}
\mathcal{O}(\overline{e}^n_{\Bu}) =&  - \frac12 [(\Bu^{n-1} \cdot \nabla \overline{\Bu}^n, \overline{e}^n_{\Bu}) - (\Bu_h^{n-1} \cdot \nabla \overline{\Bu}_h^n, \overline{e}^n_{\Bu})] + \frac12 [(\Bu^{n-1} \cdot \nabla \overline{e}^n_{\Bu},  \overline{\Bu}^n) - (\Bu_h^{n-1} \cdot \nabla \overline{e}^n_{\Bu}, &  \overline{\Bu}_h^n)]\\
 =&  - \frac12 (\Bu^{n-1} \cdot \nabla (\overline{\Bu}^n - \overline{\Bu}^n_h), \overline{e}^n_{\Bu}) - \frac12((\Bu^{n-1} - \Bu_h^{n-1}) \cdot \nabla \overline{\Bu}_h^n, \overline{e} ^n_{\Bu}) \\
&+ \frac12 (\Bu^{n-1} \cdot \nabla \overline{e}^n_{\Bu},  \overline{\Bu}^n - \overline{\Bu}^n_h) + \frac12 ((\Bu^{n-1} - \Bu_h^{n-1}) \cdot \nabla \overline{e}^n_{\Bu},  \overline{\Bu}_h^n) \\
 =&  - \frac12 (\Bu^{n-1} \cdot \nabla (\overline{\Bu}^n - \overline{\Bu}^n_h), \overline{e}^n_{\Bu}) - \frac12((\Bu^{n-1} - \Bu_h^{n-1}) \cdot \nabla \BPi_V \overline{\Bu}^n, \overline{e} ^n_{\Bu}) \\
&+ \frac12 (\Bu^{n-1} \cdot \nabla \overline{e}^n_{\Bu},  \overline{\Bu}^n - \overline{\Bu}^n_h) + \frac12 ((\Bu^{n-1} - \Bu_h^{n-1}) \cdot \nabla \overline{e}^n_{\Bu},  \BPi_V\overline{\Bu}^n).
\end{align*}
The terms in the last step can be bounded using H\"{o}lder's inequality, Sololev inequalities Lemma \ref{sobolev_embed} and the approximation properties of the projections in Lemma \ref{proj_prop} as:
\begin{align*}
- \frac12 (\Bu^{n-1} \cdot \nabla (\overline{\Bu}^n - \overline{\Bu}^n_h), \overline{e}^n_{\Bu}) & = - \frac12 (\Bu^{n-1} \cdot \nabla (\overline{\Bdelta}_{\Bu}^n + \overline{e}_{\Bu}^n), \overline{e}^n_{\Bu}) \\
&\le C \|\Bu^{n-1}\|_{0,\infty} \|\nabla(\overline{\Bdelta}_{\Bu}^n + \overline{e}_{\Bu}^n)\| \|\overline{e}_{\Bu}^n\| \\
& \le C \epsilon \|\nabla \overline{e}^n_h\|^2 + C \epsilon^{-1} \|\overline{e}^n_h\|^2 + C\epsilon^{-1} h^{2\beta}, \\
- \frac12((\Bu^{n-1} - \Bu_h^{n-1}) \cdot \nabla \BPi_V \overline{\Bu}^n, \overline{e} ^n_{\Bu}) & \le C  \|(\overline{\Bdelta}_{\Bu}^{n-1} + {e}_{\Bu}^{n-1})\| \| \nabla \BPi_V \overline{\Bu}^n\|_{0,3} \|\overline{e} ^n_{\Bu}\|_{0,6} \\
& \le C \epsilon \|\nabla \overline{e}^n_h\|^2 + C \epsilon^{-1} \|{e}^{n-1}_h\|^2 + C\epsilon^{-1} h^{2\beta}, \\
 \frac12 (\Bu^{n-1} \cdot \nabla \overline{e}^n_{\Bu},  \overline{\Bu}^n - \overline{\Bu}^n_h) & \le  \|\Bu^{n-1}\|_{0,\infty} \|\overline{\Bdelta}_{\Bu}^n + \overline{e}_{\Bu}^n\| \|\nabla \overline{e}_{\Bu}^n\| \\
 & \le C \epsilon \|\nabla \overline{e}^n_h\|^2 + C \epsilon^{-1} \|\overline{e}^n_h\|^2 + C\epsilon^{-1} h^{2\beta}, \\
  \frac12 ((\Bu^{n-1} - \Bu_h^{n-1}) \cdot \nabla \overline{e}^n_{\Bu},  \BPi_V\overline{\Bu}^n) & \le \|\BPi_V \overline{\Bu}^{n}\|_{0, \infty} \|{\Bdelta}_{\Bu}^{n-1} + {e}_{\Bu}^{n-1}\| \|\nabla \overline{e}_{\Bu}^n\| \\
 & \le C \epsilon \|\nabla \overline{e}^n_h\|^2 + C \epsilon^{-1} \|{e}^{n-1}_h\|^2 + C\epsilon^{-1} h^{2\beta}.
\end{align*}
This concludes that 
\begin{equation}\label{est_O}
\mathcal{O}(\overline{e}^n_{\Bu}) \le  C \epsilon \|\nabla \overline{e}^n_h\|^2 + C \epsilon^{-1} (\|{e}^{n-1}_h\|^2 + \|{e}^n_h\|^2) + C\epsilon^{-1} h^{2\beta}.
\end{equation}

Similarly, for $ \mathcal{M}_1(\overline{e}^n_{\Bu}) + \mathcal{M}_2 (S R^{-1}_m \nabla_h \times \overline{e}^n_{\BB})$ we start with some algebraic rearrangement as follows:
\begin{align*}
 \mathcal{M}_1(\overline{e}^n_{\Bu}) +& \mathcal{M}_2 (S R^{-1}_m \nabla_h \times \overline{e}^n_{\BB}) = \\
& S R^{-1}_m  ((\nabla_h \times (\overline{\BB}^n - \overline{\BB}^n_h) \times \BB^{n-1}, \overline{e}^n_{\Bu}) +  S R^{-1}_m   ((\nabla_h \times \overline{\BB}_h^n) \times (\BB^{n-1} -\BB_h^{n-1}), \overline{e}^n_{\Bu}) \\ 
 +& S R^{-1}_m  ((\overline{\Bu}^n - \overline{\Bu}^n_h) \times \BB^{n-1}, \nabla_h \times \overline{e}^n_{\BB}) + S R^{-1}_m  (\overline{\Bu}^n_h \times (\BB^{n-1} - \BB^{n-1}_h), \nabla_h \times \overline{e}^n_{\BB}) \\
 = & M_1 + M_2 + M_3 + M_4.
\end{align*}
Next we will estimate $M_1 + M_3$ and $M_2 + M_4$ separately. Namely, we have
\begin{align*}
M_1 + M_3 &= S R^{-1}_m  ((\nabla_h \times (\overline{\Bdelta}_{\BB}^n + \overline{e}^n_{\BB}) \times \BB^{n-1}, \overline{e}^n_{\Bu}) +  S R^{-1}_m  ((\overline{\Bdelta}_{\Bu}^n + \overline{e}_{\Bu}^n) \times \BB^{n-1}, \nabla_h \times \overline{e}^n_{\BB}) \\
&= S R^{-1}_m  ((\nabla_h \times \overline{\Bdelta}_{\BB}^n) \times \BB^{n-1}, \overline{e}^n_{\Bu}) +  S R^{-1}_m  (\overline{\Bdelta}_{\Bu}^n  \times \BB^{n-1}, \nabla_h \times \overline{e}^n_{\BB}) \\
& +  S R^{-1}_m  ((\nabla_h \times \overline{e}^n_{\BB}) \times \BB^{n-1}, \overline{e}^n_{\Bu}) +  S R^{-1}_m  ( \overline{e}_{\Bu}^n \times \BB^{n-1}, \nabla_h \times \overline{e}^n_{\BB}) \\
\intertext{The last two terms cancelled out due to the fact $\Ba \times \Bb + \Bb \times \Ba = 0$, the first term vanishes due to the fact $\nabla_h \times \Bdelta^n_{\BB} = \boldsymbol{0}$, hence}
M_1 + M_3 &= S R^{-1}_m  (\overline{\Bdelta}_{\Bu}^n  \times \BB^{n-1}, \nabla_h \times \overline{e}^n_{\BB}) \le C \|\overline{\Bdelta}^n_{\Bu}\|_{0, 6}\|\BB^{n-1}\|_{0,3}   \|\nabla_h \times \overline{e}_{\BB}^n\| \\
& \le C \epsilon \|\nabla_h \times \overline{e}_{\BB}^n\|^2 + C \epsilon^{-1} h^{2\beta}.
\end{align*}
For $M_2 + M_4$, we insert this identity:
\[
S R^{-1}_m   ((\nabla_h \times \overline{e}_{\BB}^n) \times (\BB^{n-1} -\BB_h^{n-1}), \overline{e}^n_{\Bu}) +  S R^{-1}_m  (\overline{e}_{\Bu}^n \times (\BB^{n-1} - \BB^{n-1}_h), \nabla_h \times \overline{e}^n_{\BB}) = 0
\]
into $M_2 + M_4$ with simple algebraic cancelation, we arrive at:
\begin{align*}
M_2 &+ M_4 \\
&= S R^{-1}_m  \left[ ((\nabla_h \times \BPi_D \overline{\BB}^n) \times (\BB^{n-1} -\BB_h^{n-1}), \overline{e}^n_{\Bu}) +   (\BPi_V \overline{\Bu}^n \times (\BB^{n-1} - \BB^{n-1}_h), \nabla_h \times \overline{e}^n_{\BB}) \right] \\
&\le C \|(\nabla_h \times \BPi_D \overline{\BB}^n\|_{0,3} \|\Bdelta^{n-1}_{\BB} + e^{n-1}_{\BB}\| \|\overline{e}^n_{\Bu}\|_{0,6} + \|\BPi_V \overline{\Bu}^n\|_{0, \infty} \|\Bdelta^{n-1}_{\BB} + e^{n-1}_{\BB}\| \| \nabla_h \times \overline{e}^n_{\BB}\| \\
& \le C \epsilon (\|\nabla \overline{e}^n_{\Bu}\|^2 + \| \nabla_h \times \overline{e}^n_{\BB}\|^2)  + C \epsilon^{-1} \|e^{n-1}_{\BB}\|^2 + C \epsilon^{-1} h^{2\beta}.
\end{align*}

Now if we combine all the above estimates we arrive at:
\begin{align}\label{energy_tn}
\frac{\|e^n_{\Bu}\|^2 - \|e^{n-1}_{\Bu}\|^2}{2 \tau} &+ R^{-1}_e \|\nabla \overline{e}^n_{\Bu}\|^2 + S R^{-1}_m \frac{\|e^n_{\BB}\|^2 - \|e^{n-1}_{\BB}\|^2}{2 \tau} + S R^{-2}_m \|\nabla_h \times \overline{e}^n_{\BB}\|^2 \\
\nonumber
& \hspace{-1cm}\le C \epsilon ( \|\nabla \overline{e}^n_{\Bu}\|^2 + \|\nabla_h \times \overline{e}^n_{\BB}\|^2) + C \epsilon^{-1} (h^{2s} + \|{e}^n_{\Bu}\|^2 + \|\overline{e}^{n-1}_{\Bu}\|^2 + \|{e}^n_{\BB}\|^2 + \|{e}^{n-1}_{\BB}\|^2) \\
\nonumber
 & \hspace{-1cm}+ C \tau (\|\Bu_{tt}\|^2_{L^2(t_{n-1}, t_n; \BL^2(\Omega))} + \|\BB_{tt}\|^2_{L^2(t_{n-1}, t_n; \BL^2(\Omega))})\\
\nonumber
 & \hspace{-1cm}+ C \epsilon^{-1} \tau (\|\Bu_t\|^2_{L^2(t_{n-1}, t_n, \BH^1(\Omega))} + \|\BB_t\|^2_{L^2(t_{n-1}, t_n, \BL^2(\Omega))} + \|\nabla \times \BB_t\|^2_{L^2(t_{n-1}, t_n, \BL^2(\Omega))}) \\
\nonumber
 &\hspace{-1cm}+ C h^{2\beta} \tau^{-1} \|\Bu_t\|^2_{L^2(t_{n-1}, t_n; \BH^s(\Omega))}.
\end{align}
If we take $\epsilon = \min\{\frac12 R^{-1}_e, \frac12 S R^{-2}_m \}$, multiplying $2 \tau$ on the above estimate and sum over $j = 1, \cdots, n$ we have 
\begin{align*}
\|e^n_{\Bu}\|^2 + \|e^n_{\BB}\|^2 &+ C \tau \sum_{j = 1}^n ( \|\nabla \overline{e}^n_{\Bu}\|^2 + \|\nabla_h \times \overline{e}^n_{\BB}\|^2) \\
& \le \|e^0_{\Bu}\|^2 + \|e^0_{\BB}\|^2 + C \tau \sum_{j=0}^n (h^{2s} + \|e^n_{\Bu}\|^2 + \|e^n_{\BB}\|^2) \\
&+ C \tau^2 (\|\Bu_{tt}\|^2_{L^2(0, t_n; \BL^2(\Omega))} + \|\BB_{tt}\|^2_{L^2(0, t_n; \BL^2(\Omega))})\\
& + C \tau^2 (\|\Bu_t\|^2_{L^2(0, t_n; \BH^1(\Omega))} + \|\BB_t\|^2_{L^2(0, t_n; \BL^2(\Omega))} + \|\nabla \times \BB_t\|^2_{L^2(0, t_n; \BL^2(\Omega))}) \\
& + C h^{2\beta} \|\Bu_t\|^2_{L^2(0, t_n; \BH^\beta(\Omega))}.
\end{align*}
\\
By the fact $e^0_{\Bu} = 0, e^0_{\BB} = 0$ and the regularity assumption \eqref{regularity}, we have
\begin{align*}
\|e^n_{\Bu}\|^2 + \|e^n_{\BB}\|^2 &+ C \tau \sum_{j = 1}^n ( \|\nabla \overline{e}^n_{\Bu}\|^2 + \|\nabla_h \times \overline{e}^n_{\BB}\|^2) \\
& \le C \tau \sum_{j=0}^n ( \|e^n_{\Bu}\|^2 + \|e^n_{\BB}\|^2)  + C (h^{2\beta} + \tau^2).
\end{align*}
By the discrete Gronwall's inequality Lemma \ref{discrete_gronwall} with $C \tau < \frac12$, we have
\begin{align*}
\|e^n_{\Bu}\|^2 + \|e^n_{\BB}\|^2 &+ C \tau \sum_{j = 1}^n ( \|\nabla \overline{e}^n_{\Bu}\|^2 + \|\nabla_h \times \overline{e}^n_{\BB}\|^2) \\
& \le e^{2CT} (h^{2\beta} + \tau^2).
\end{align*}
We complete the proof of \eqref{energy_all} by applying the triangle inequality, approximation properties of the projections Lemma \ref{proj_prop} together with above etimates. Combine the above estimate with \eqref{energy_tn} we can deduce the estimates \eqref{est_Bu_tn} with the initial error estimates \eqref{initial_eror}. 

With the above estimates for $\Bu, \BB$, we can simply take $\BF = e_{\BE}^n$ in \eqref{error_eq2}, with some algebraic rearrangement, we arrive at:
\begin{align*}
\|e^n_{\BE}\|^2 & = R^{-1}_m (\nabla_h \times \overline{e}^n_{\BB}, e^n_{\BE}) - (\Bdelta_{\BE}^n, e^n_{\BE}) + \mathcal{R}_2(e^n_{\BE}) - \mathcal{M}_2(e^n_{\BE}) \\
& \le C (\|\nabla_h \times \overline{e}^n_{\BB}\| + \|\Bdelta_{\BE}^n\|) \|e^n_{\BE}\| + \mathcal{R}_2(e^n_{\BE}) - \mathcal{M}_2(e^n_{\BE}).
\end{align*}
For the last two terms, with a similar treatment as in the previous proofs, we can bound these two terms as follows: 
\begin{align*}
\mathcal{R}_2(e^n_{\BE}) & = ((\overline{\Bu}^n - \Bu^n) \times \BB^{n-1}, e^n_{\BE}) + (\Bu^n \times (\BB^{n-1} - \BB^n), e^n_{\BE}) - R^{-1}_m (\nabla \times (\overline{\BB}^n - \BB^n), e^n_{\BE}) \\
& \le C \tau^{\frac12} \|e^n_{\BE}\| ( \|\Bu_t\|_{L^{2}(t_{n-1}, t_n; \BH^1(\Omega))} \|\BB^{n-1}\|_{0,3}  + \|\Bu^n\|_{0, \infty} \|\BB_t\|_{L^{2}(t_{n-1}, t_n; \BL^2(\Omega))}) \\
 & + C \tau^{\frac12} \|e^n_{\BE}\|  \|\nabla \times \BB_t\|_{L^{2}(t_{n-1}, t_n; \BL^2(\Omega))}), \\
 \mathcal{M}_2(e^n_{\BE}) & = ((\overline{\Bu}^n - \overline{\Bu}^n_h) \times \BB_h^{n-1}, e^n_{\BE}) + (\overline{\Bu}^n \times (\BB^{n-1} - \BB^{n-1}_h), e^n_{\BE}) \\
 & \le C \|e^n_{\BE}\| (\|\overline{\Bu}^n - \overline{\Bu}^n_h\|_{0,6} \|\BB^{n-1}_h\|_{0,3} + \|\Bu^n\|_{\infty} \|(\BB^{n-1} - \BB^{n-1}_h\|) \\
 & \le C  \|e^n_{\BE}\| (\|\nabla (\overline{\Bu}^n - \overline{\Bu}^n_h)\| \|\nabla_h \times \BB^{n-1}_h\| + \|\Bu^n\|_{1 + s} \|(\BB^{n-1} - \BB^{n-1}_h\|) \\
 &\le C h^{\beta} \|e^n_{\BE}\|.
\end{align*}
 The above estimates is due to Sobolev inequality Lemma \ref{sobolev_embed} and Theorem 1 in \cite{HuQiuShi19}, stability of the solution Theorem \ref{stability} and the estimates for $\Bu, \BB$. This completes the proof for \eqref{estimate_e1} with a simple triangle inequality and projection error estimates for $\BE$ in Lemma \ref{proj_prop}.
 
 With a slightly stronger regularity assumption with $\Bu_t \in L^{\infty}(0, T; \BH^1(\Omega)), \BB_t, \nabla \times \BB_t \in L^{\infty}(0, T; \BL^2(\Omega))$ we can regain the full order of $\tau$ as: 
\begin{align*}
\mathcal{R}_2(e^n_{\BE}) & = ((\overline{\Bu}^n - \Bu^n) \times \BB^{n-1}, e^n_{\BE}) + (\Bu^n \times (\BB^{n-1} - \BB^n), e^n_{\BE}) - R^{-1}_m (\nabla \times (\overline{\BB}^n - \BB^n), e^n_{\BE}) \\
& \le C \tau \|e^n_{\BE}\| ( \|\Bu_t\|_{L^{\infty}(t_{n-1}, t_n; \BH^1(\Omega))} \|\BB^{n-1}\|_{0,3}  + \|\Bu^n\|_{0,{\infty}(\Omega)} \|\BB_t\|_{L^{\infty}(t_{n-1}, t_n; \BL^2(\Omega))}) \\
 & + C \tau \|e^n_{\BE}\|  \|\nabla \times \BB_t\|_{L^{\infty}(t_{n-1}, t_n; \BL^2(\Omega))}), 
\end{align*}
this completes the proof for \eqref{estimate_e2}.

Finally we use a classical {\em inf-sup} argument to bound $e_p$ as in \eqref{estimate_p}. By the {\em inf-sup} condition \eqref{infsup} we know that there exists $\Bw_h \in \BV_h$ such that
\begin{equation}\label{error_p0}
 \|e_p^n\| \le \frac{1}{\kappa} \frac{(e_p^n, \nabla \cdot \Bw_h)}{\|\Bw_h\|_1}.
\end{equation}
On the other hand, by error equation \eqref{error_eq1} we have
\begin{equation}\label{error_p1}
(e_p^n, \nabla \cdot \Bw_h) = (D_{\tau} e^n_{\Bu}, \Bw_h) + (D_{\tau} \Bdelta^n_{\Bu}, \Bw_h) + R_e^{-1} (\nabla \overline{e}_\Bu^n, \nabla \Bw_h) - \mathcal{R}_1(\Bw_h) - \mathcal{O}(\Bw_h) - \mathcal{M}_1(\Bw_h).
\end{equation}
Each of the terms on the right hand side can be estimated as follows:
\begin{align*}
(D_{\tau} \Bdelta^n_{\Bu}, \Bw_h) &= \frac{1}{\tau} \int_{\Omega}\int_{t_{n-1}}^{t_n} \frac{\partial \Bdelta_{\Bu}}{\partial t} \cdot \Bw_h dt d \Bx = \frac{1}{\tau} \int_{t_{n-1}}^{t_n} \int_{\Omega} \frac{\partial \Bdelta_{\Bu}}{\partial t} \cdot \Bw_h d \Bx dt , \\
& \le  \frac{1}{\tau} \int_{t_{n-1}}^{t_n} \|\frac{\partial \Bdelta_{\Bu}}{\partial t} \| \|\Bw_h\| d\rho \le C \frac{1}{\tau} \int_{t_{n-1}}^{t_n} h^{\beta} \|\Bu_t(\rho, \cdot)\|_{\beta} d\rho \\
& \le C \tau^{-\frac12} h^{\beta} \|\Bu_t\|_{L^{2}(t_{n-1}, t_n; \BH^{\beta}(\Omega))}\|\Bw_h\|.\\
 R_e^{-1} (\nabla \overline{e}_\Bu^n, \nabla \Bw_h) &\le C \|\nabla \overline{e}_{\Bu}^n\| \|\Bw_h\|_1 \le C (\tau + h^{\beta}) \|\Bw_h\|_1, 
\end{align*}
For $\mathcal{R}_1(\Bw_h)$, with a similar estimates for the terms in $\mathcal{R}_1(\overline{e}_{\Bu}^n)$, we have:
\[
\mathcal{R}_1(\Bw_1) \le C \tau^{\frac12} \|\Bw_h\|_1.
\]
Notice the above result is slightly different from the estimates for $\mathcal{R}_1(\overline{e}^n_{\Bu})$ due to the fact that we don't apply the weighted Young's inequality here for each term. For instance, 
\[
R^{-1}_e (\nabla \overline{\Bu}^n - \nabla \Bu^n, \nabla \Bw_h) \le C \tau^{\frac12} \|\nabla \Bu_t\|_{L^2(t_{n-1}, t_n; \BL^2(\Omega))} \| \Bw_h\|_1.
\]
Similarly, for $\mathcal{O}(\Bw_h)$ we have:
\begin{align*}
\mathcal{O}(\Bw_h) \le C(\tau + h^{\beta}) \|\Bw_h\|_1, 
\end{align*}
For $\mathcal{M}_1(\Bw_h)$, we have
\begin{align*}
\mathcal{M}_1(\Bw_h) &= S R^{-1}_m((\nabla_h \times \overline{\BB}^n) \times \BB^{n-1}, \Bw_h) - S R^{-1}_m((\nabla_h \times \overline{\BB}_h^n) \times \BB_h^{n-1}, \Bw_h) \\
& = S R^{-1}_m((\nabla_h \times \overline{\BB}^n) \times (\BB^{n-1} - \BB^{n-1}_h), \Bw_h) + S R^{-1}_m((\nabla_h \times (\overline{\BB}^n - \overline{\BB}_h^n)) \times \BB_h^{n-1}, \Bw_h) \\
&\le C \|\nabla_h \times \overline{\BB}^n\|_{0,3} \|\BB^{n-1} - \BB^{n-1}_h\| \|\Bw_h\|_{0,6} + C \|\nabla_h \times (\overline{\BB}^n - \overline{\BB}_h^n)\| \|\BB^{n-1}_h\|_{0,3} \|\Bw_h\|_{0,6} \\
&\le C \|\nabla \times \overline{\BB}^n\|_{0,3} \|\BB^{n-1} - \BB^{n-1}_h\| \|\Bw_h\|_{1} + C \|\nabla_h \times (\overline{\BB}^n - \overline{\BB}_h^n)\| \|\nabla_h \times \BB^{n-1}_h\| \|\Bw_h\|_1 \\
& \le C h^{\beta} \|\Bw_h\|_1.
\end{align*}
For the last term we start with Cauchy-Schwarz inequality to have:
\begin{align*}
 (D_{\tau} e^n_{\Bu}, \Bw_h) &\le \|D_{\tau}e^n_{\Bu}\| \|\Bw_h\|.
\end{align*}
If now we directly bound $\|D_{\tau} e^n_{\Bu}\| \le \tau^{-1} (\|e^n_{\Bu} + e^{n-1}_{\Bu}\|$ we will lose a full power of $\tau$ which means there is no convergence order in time for $e_p$. In stead, we take $\Bv = D_{\tau} e^n_{\Bu}$ in \eqref{error_eq1} to have, 
\[
\|D_{\tau} e^n_{\Bu}\|^2 = - (D_{\tau} \Bdelta^n_{\Bu}, \Bw_h) - R_e^{-1} (\nabla \overline{e}_\Bu^n, \nabla D_{\tau} e^n_{\Bu}) + \mathcal{R}_1(D_{\tau} e^n_{\Bu}) + \mathcal{O}(D_{\tau} e^n_{\Bu}) + \mathcal{M}_1(D_{\tau} e^n_{\Bu}).
\]
Here we used the fact that $(\nabla \cdot e^n_{\Bu}, q) = 0$ for all $n$ due to the error equation \eqref{error_eq4}. The second term on the right hand side can be bounded as:
\[
- R_e^{-1} (\nabla \overline{e}_\Bu^n, \nabla D_{\tau} e^n_{\Bu}) = - R_e^{-1} (2\tau)^{-1} (\|\nabla e^n_{\Bu}\|^2 - \|\nabla e^{n-1}_{\Bu}\|^2) \le C (\tau + \tau^{-1} h^{2\beta}).
\]
For the rest terms on the right hand side, we bound them in the same way as above, after simplification, we arrive at:
\[
\|D_{\tau} e^n_{\Bu}\|^2 \le C \tau^{-\frac12}h^{\beta} \|D_{\tau} e^n_{\Bu}\| + C (\tau + h^{\beta})\|D_{\tau} e^n_{\Bu}\| + C (\tau + \tau^{-1} h^{2\beta}).
\]
This implies that 
\[
\|D_{\tau} e^n_{\Bu}\| \le C (\tau^{\frac12} + \tau^{-\frac12} h^{\beta}).
\]
Finally if we combine all the above estimates into \eqref{error_p0}, \eqref{error_p1}, we finally have:
\[
\|e^n_p\|^2 \le C (\tau^2 + \tau^{-1} h^{2\beta}).
\]
This completes all the estimates in Theorem \ref{main_error}.
\end{proof}


\bibliographystyle{plain}      

\begin{thebibliography}{10}


\bibitem{ArnoldFalkWinther06}
D.N. Arnold, R. S Falk, and R. Winther, {\em Finite element exterior calculus, homological techniques, and applications.} Acta numerica, 15:1-155, 2006.

\bibitem{ArnoldFalkWinther10}
D.N. Arnold, R. S Falk, and R. Winther, {\em Finite element exterior calculus: from hodge theory to numerical stability.} Bulletin of the American mathematical society, 47(2):281-354, 2010.




\bibitem{Abdou2001}
{\sc M.A. Abdou et al.} (2001), {On the exploration of innovative concepts for fusion chamber technology}, 
{\it Fusion Eng. Des.}, 54, pp. 181--247.

\bibitem{Badia-Codina-Planas}  
{\sc S.~Badia, R.~Codina and R.~Planas} (2013).
{On an unconditionally convergent stabilized finite element approximation of resistive magnetohydrodynamics}, 
{\it J. Comput. Phys.}, 234, pp. 399--416.

\bibitem{BanasProhl2010}
{\sc L. Ba{\v{n}}as and A. Prohl} (2010). {Convergent finite element discretization of the multi-fluid nonstationary 
incompressible magnetohydrodynamics equations}, {\it Mathematics of Computation}, 79(272), pp. 1957--1999.

\bibitem{Codina}
{\sc R.~Codina and N.~Hern\'andez} (2011). 
{Approximation of the thermally coupled MHD problem using a stabilized finite element method},
{\it J. Comput. Phys.},230, pp. 1281--1303.

\bibitem{CT1987}
{\sc M. Crouzeix and V. Thom{\'{e}}e} (1987). {The Stability in $L_{p}$ and $W_{p}^{1}$ of the 
$L_{2}$-Projection onto Finite Element Function Spaces}, 
{\it Mathematics of Computation}, 48(178), pp. 521--532.

\bibitem{GaoQiu2019}
{\sc H. Gao and W. Qiu} (2019). {A semi-implicit energy conserving finite element method for the dynamical 
incompressible magnetohydrodynamics equations}, {\it Computer Methods in Applied Mechanics and 
Engineering}, 346, pp. 982--1001.

\bibitem{Gerbeau-LL}  
{\sc  J.~Gerbeau, C.~Le Bris and T.~Leli\`evre} (2006).
{Mathematical Methods for the Magnetohydrodynamics of Liquid Metals}, 
{\it Oxford University Press, Oxford}.

\bibitem{Girault.V;Raviart.P.1986a}
V. Girault and P-A. Raviart.
\newblock {\em Finite element methods for {Navier-Stokes} equations: theory and
  algorithms}, volume~5.
\newblock Springer Science \& Business Media, 2012.
  
\bibitem{Goedbloed-Poedts}
{\sc J.~Goedbloed and S.~Poedts} (2004).
{Principles of Magnetohydrodynamics with Applications to Laboratory and Astrophysical Plasmas},
{\it Cambridge University Press, Cambridge, MA}.


\bibitem{He2015}
{\sc Y. He} (2015). {Unconditional convergence of the Euler semi-implicit scheme for the 
three-dimensional incompressible MHD equations}, 
{\it IMA Journal of Numerical Analysis}, 35(2), pp. 767--801.

\bibitem{HeywoodRannacher1990}
{\sc J. Heywood and R. Rannacher} (1990). {Finite element approximation of the nonstationary Navier-Stokes 
problem IV: Error analysis for second-order time discretization},
{\it  SIAM J. Numer. Anal.}, 27, pp. 353-384.

\bibitem{Hiptmair2002}
R. Hiptmair, {\em Finite elements in computational electromagnetism,} Acta. Numer., 11 (2002), pp. 237-339.

\bibitem{Hiptmair_LMZ}  
{\sc R.~Hiptmair, M.~Li, S.~Mao and W.~Zheng} (2018).
{A fully divergence-free finite element method for magnetohydrodynamic equations}, 
{\it Math. Models Methods Appl. Sci.}, 28, pp. 659--695.

\bibitem{HuMaXu}
{\sc K.~Hu, Y.~Ma and J.~Xu} (2017).
{Stable finite element methods preserving $\nabla \cdot \mathbf{B}=0$ exactly for MHD models},
{\it Numer. Math.}, 135, pp. 371--396.

\bibitem{MaXuZhang2016}
{\sc Y. Ma, J. Xu and G. Zhang} (2016). {Error estimates for structure-preserving discretization of 
the incompressible MHD system}, arXiv:1608.03034.

\bibitem{Marioni}  
{\sc L.~Marioni, F.~Bay and E.~Hachem} (2016).
{Numerical stability analysis and flow simulation of lid-driven cavity subjected to high magnetic field},
{\it Phys. Fluids}, 28, pp. 57--102.

\bibitem{Moreau}
{\sc R.~Moreau} (1990). {Magnetohydrodynamics}, {\it Kluwer Academic Publishers, New York}.

\bibitem{Ni2007a}
{\sc M.-J. Ni, R. Munipalli, P. Huang, N.B. Morley, M.A. Abdou} (2007). {A current density con-
servative scheme for incompressible MHD flows at a low magnetic Reynolds number. Part I. 
On a rectangular collocated grid system}, {\it J. Comp. Phys.}, 227, pp. 174--204.

\bibitem{Ni2007b}
{\sc M.-J. Ni, R. Munipalli, P. Huang, N.B. Morley, M.A. Abdou} (2007). {A current density conservative scheme 
for incompressible MHD flows at a low magnetic Reynolds number. Part II: On an arbitrary collocated mesh},
{\it  J. Comp. Phys.}, 227, pp. 205--228.

\bibitem{Nirenberg1966}
L. Nirenberg, {\em An extended interpolation inequality,} Ann. Scuola Norm. Sup. Pisa (3), 20(1966), pp. 733-737.


\bibitem{Monk2003}
P. Monk, {\em Finite element methods for Maxwell's equations.} Oxford University Press, New York, 2003.

\bibitem{Phillips}  
{\sc E.~Phillips, H.~Elman, E.~Cyr, J.~Shadid and R.~Pawlowski} (2016).
{Block preconditioners for stable mixed nodal and
  edge finite element representations of incompressible resistive MHD},
{\it SIAM J. Sci. Comput.}, 36, pp. B1009--B1031.

\bibitem{Prohl}  
{A.~Prohl} (2008). 
{\sc Convergent finite element discretizations of the nonstationary
  incompressible magnetohydrodynamics system},
{\it M2AN Math. Model. Numer. Anal.}, 42, pp. 1065--1087.

\bibitem{QiuShi2019}
W. Qiu and K. Shi, {\em Analysis on an HDG Method for the p-Laplacian Equations}, J Sci Comput (2019) 80: 1019. DOI
https://doi.org/10.1007/s10915-019-00967-6.

\bibitem{HuQiuShi19}
K. Hu, W. Qiu and K. Shi {\em Convergence of a B-E based finite element method for MHD models on Lipschitz domains}. J. Comput. and Appl. Math. DOI https://doi.org/10.1016/j.cam.2019.112477

\bibitem{Sermane-Temam}
{\sc M.~Sermane and R.~Temam} (1984).
{Some mathematics questions related to the MHD equations},
{\it Commun. Pure Appl. Math.}, XXXIV, pp. 635--664.

\bibitem{Ni2013}
{\sc S. Xu, N. Zhang and M. Ni} (2013), {Influence of flow channel insert with pressure equalization opening 
on MHD flows in a rectangular duct}, {\it Fusion Eng. Des.}, 88, pp. 271--275. 
  
\bibitem{Zhang-Yang-Bi}
{\sc G.~Zhang, J.~Yang and C.~Bi} (2018).
{Second order unconditionally convergent and energy stable linearized scheme for MHD equations},
{\it Adv. Comput. Math.}, 44, pp. 505--540.  

\bibitem{Ni2014}
{\sc J. Zhang and M. Ni} (2014). {A consistent and conservative scheme for MHD flows with complex 
boundaries on an unstructured Cartesian adaptive system}, {\it J. Comp. Phys.}, 256, pp. 520--542.


\end{thebibliography}

\end{document}